\documentclass[11pt]{article}

\usepackage[affil-it]{authblk}

\usepackage{amsmath,amsfonts,amsthm,amssymb,verbatim,graphicx,color}

\def\e{{\mathbf e}}
\def\x{{\mathbf x}}
\def\y{{\mathbf y}}
\newcommand{\uno}{{\mathbf{1}}}

\newcommand{\R}{\mathbf{R}}

\newcommand{\thbo}{{\rm TH}}
\newcommand{\stab}{{\rm STAB}}
\newcommand{\estab}{{\rm ESTAB}}
\newcommand{\fra}{{\rm ESTAB}}
\newcommand{\qstab}{{\rm QSTAB}}
\newcommand{\astab}{{\rm ASTAB}}
\newcommand{\rstab}{{\rm RSTAB}}

\newcommand{\conv}{\mathrm{conv}}

\DeclareMathOperator{\cone}{cone}
\DeclareMathOperator{\diag}{diag}

\newcommand{\set}[1]{\left\{#1\right\}}

\newtheorem{theorem}{Theorem}
\newtheorem{lemma}[theorem]{Lemma}

\newtheorem{corollary}[theorem]{Corollary}
\newtheorem{conjecture}[theorem]{Conjecture}

\begin{document}
\title{Characterizing $N_+$-perfect line graphs}

\author[1]{Mariana Escalante}
\author[1]{Graciela Nasini}
\author[2]{Annegret Wagler}
\affil[1]{CONICET and FCEIA, Universidad Nacional de Rosario\\
Rosario, Argentina\\
$\{$mariana,nasini$\}$@fceia.unr.edu.ar}
\affil[2]{CNRS and LIMOS, Universit\'e Blaise Pascal\\ 
Clermont-Ferrand, France\\
wagler@isima.fr}

\maketitle

\begin{abstract}
\boldmath
The subject of this contribution is the study of the Lov\'asz-Schrijver PSD-operator $N_+$ applied to the edge relaxation of the stable set polytope of a graph. 
We are particularly interested in the problem of characterizing graphs for which $N_+$ generates the stable set polytope in one step, called $N_+$-perfect graphs. 
It is conjectured that the only $N_+$-perfect graphs are those whose stable set polytope is described by inequalities with near-bipartite support. 
So far, this conjecture has been proved for near-perfect graphs, fs-perfect graphs, and webs.
Here, we verify it for line graphs, by proving that in an $N_+$-perfect line graph the only facet-defining 
graphs are cliques and odd holes.
\end{abstract}


\section{Introduction}

The context of this paper is the study of the stable set polytope, different of its linear and semi-definite relaxations, and graph classes for which certain relaxations are tight. 
Our focus lies on $N_+$-perfect graphs: those graphs where a single application of the Lov\'asz-Schrijver PSD-operator $N_+$ to the edge relaxation yields the stable set polytope.

The \textit{stable set polytope} $\stab(G)$ of a graph $G=(V,E)$ is defined 
as the convex hull of the incidence vectors of all stable sets of $G$ 
(in a stable set all nodes are mutually nonadjacent). 

Two canonical relaxations of $\stab(G)$ are the
\textit{fractional} or \textit{edge constraint stable set polytope} 
$$
\estab(G) = \{\mathbf{x} \in \R_+^{|V|}: x_i + x_j \: \le \: 1, ij \in E \}, 
$$
and the \textit{clique constraint stable set polytope} 
$$
\qstab(G) = \{\mathbf{x} \in \R_+^{|V|}: \sum_{i \in Q} \, x_i \: \le \: 1,\ Q \subseteq V\ \mbox{clique} \}
$$
(in a clique all nodes are mutually adjacent, hence a clique and a stable set share at most one node). We have 
$$
\stab(G) \subseteq \qstab(G) \subseteq \estab(G)
$$ 
for any graph, where $\stab(G)$ equals $\estab(G)$ for bipartite graphs, and $\qstab(G)$ for perfect graphs only \cite{Chvatal1975}.  

According to a famous characterization achieved by Chudnovsky et
al.~\cite{ChudnovskyEtAl2006}, perfect graphs are precisely the graphs without
chordless cycles $C_{2k+1}$ with $k \geq 2$, termed \textit{odd holes}, or their complements, the \textit{odd antiholes} $\overline C_{2k+1}$ (the complement $\overline G$ has the same nodes as $G$, but two nodes are adjacent in $\overline G$ if and only if they are non-adjacent in $G$).  

Perfect graphs turned out to be an interesting and important class with a rich structure and a nice algorithmic behaviour \cite{GrotschelLovaszSchrijver1988}.  
However, solving the stable set problem for a perfect graph $G$ by maximizing a linear objective function over $\qstab(G)$ does not work directly \cite{GrotschelLovaszSchrijver1981}, but only via a detour involving a geometric representation of graphs \cite{Lovasz1979} and the resulting \textit{theta-body} $\thbo(G)$ introduced by Lov\'asz et al.~\cite{GrotschelLovaszSchrijver1988}. 

An orthonormal representation of a graph $G=(V,E)$ is a sequence $(\mathbf{u_i} : i \in V)$ of $|V|$ unit-length vectors $\mathbf{u_i} \in \R^N$, 
such that $\mathbf{u_i}^T\mathbf{u_j} = 0$ for all $ij \not\in E$. For any orthonormal representation of $G$ and any additional unit-length vector $\mathbf{c} \in \R^N$, the corresponding orthonormal representation constraint is $\sum_{i \in V} (\mathbf{c}^T\mathbf{u_i})^2 x_i \leq 1$. $\thbo(G)$ denotes the convex set of all vectors $\mathbf{x} \in \R_+^{|V|}$ satisfying all orthonormal representation constraints for $G$. 
For any graph $G$, we have 
\[
\stab(G) \subseteq \thbo(G) \subseteq \qstab(G)
\]
and that approximating a linear objective function over $\thbo(G)$ can be done with arbitrary precision in polynomial time \cite{GrotschelLovaszSchrijver1988}.  
Most notably, the same authors proved a beautiful characterization of perfect graphs: 
\begin{equation} \label{equivalencias}
\begin{array}{rcl}
G \textrm{ is perfect} & \Leftrightarrow & \thbo(G)=\stab(G) \\
                       & \Leftrightarrow & \thbo(G)=\qstab(G) \\
                       & \Leftrightarrow & \thbo(G) \textrm{ is polyhedral} \\
\end{array}
\end{equation}
which even shows that optimizing a linear function on polyhedral $\thbo(G)$ can be done in polynomial time.

For all imperfect graphs $G$ it follows that $\stab(G)$ does not coincide with any of the above relaxations.  
It is, thus, natural to study further relaxations and to combinatorially characterize those graphs where $\stab(G)$ equals one of them.

\subsection{A linear relaxation and rank-perfect graphs} 

Rank-perfect graphs are introduced in \cite{Wagler2000} in order to obtain a superclass of perfect graphs in terms of a further linear relaxation of $\stab(G)$. 
As natural generalization of the clique constraints describing $\qstab(G)$, we consider rank constraints 
\[
\mathbf{x}(G') = \sum_{i \in G'} \, x_i \, \le \, \alpha(G') 
\] 
associated with arbitrary induced subgraphs $G' \subseteq G$. 
By the choice of the right hand side $\alpha(G')$, denoting the size of a largest stable set in $G'$, rank constraints are obviously valid for $\stab(G)$. The \textit{rank constraint stable set polytope} 
\[
\rstab(G) = \{ \mathbf{x} \in \R^{|V|} : \sum_{i \in G'} \, x_i \: \le \: \alpha(G'),\ G' \subseteq G \}
\]
is a further linear relaxation of $\stab(G)$. 
As clique constraints are special rank constraints (namely exactly those with $\alpha(G')=1$), we immediately obtain 
\[
\stab(G) \subseteq \rstab(G) \subseteq \qstab(G).
\]
A graph $G$ is \textit{rank-perfect} by \cite{Wagler2000} if and only if $\stab(G) = \rstab(G)$ holds. By definition, rank-perfect graphs include all perfect graphs (where rank constraints associated with cliques suffice). 
In general, by restricting the facet set to rank constraints associated with certain subgraphs only, several well-known graph classes are defined, e.g., 
\textit{near-perfect graphs} \cite{Shepherd1994} where only rank constraints associated with cliques and the whole graph are  allowed, or \textit{t-perfect} \cite{Chvatal1975} and \textit{h-perfect graphs} \cite{GrotschelLovaszSchrijver1988} where rank constraints associated with edges and odd cycles resp. cliques and odd holes are required.  

Further classes of rank-perfect graphs are antiwebs~\cite{Wagler2004_4OR} and line graphs~\cite{Edmonds1965,EdmondsPulleyblank1974}. 

An \textit{antiweb} $A^k_n$ is a graph with $n$ nodes $0, \ldots, n-1$ 
and edges $ij$ if and only if $k \leq |i-j| \leq n-k$ and $i \neq j$. 
Antiwebs include all 
complete graphs $K_n = A^1_n$, 
all odd holes $C_{2k+1} = A^k_{2k+1}$, and their complements $\overline C_{2k+1} = A^2_{2k+1}$, 
see Fig. \ref{grafos_antiwebs} for examples.
\begin{figure}
\begin{center}
\includegraphics[scale=0.8]{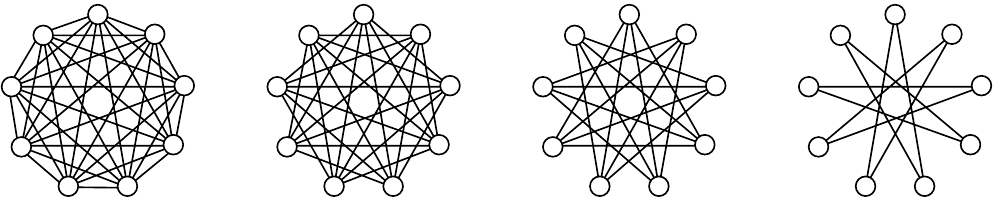}
\caption{The antiwebs $A^k_9$.}
\label{grafos_antiwebs}
\end{center}
\end{figure}

As common generalization of perfect, t-perfect and h-perfect graphs 
as well as antiwebs, the class of \textit{a-perfect graphs} was introduced 
in \cite{Wagler2005_4OR} 
as those graphs whose stable set polytopes are given by 
nonnegativity constraints and rank constraints associated with antiwebs only. 
Antiwebs are $a$-perfect by \cite{Wagler2004_4OR}, further examples of $a$-perfect graphs were presented in \cite{Wagler2005_4OR}.

A \textit{line graph} is obtained by taking the edges of a graph as nodes and 
connecting two nodes if and only if the corresponding edges are incident, 
see Fig. \ref{grafos_line} for illustration. 
Since matchings of the original graph 
correspond to stable sets of the line graph, 
the results on the matching polytope by~\cite{Edmonds1965,EdmondsPulleyblank1974} imply 
that line graphs are also rank-perfect, see Section 3 for details.
\begin{figure}[h]
\begin{center}
\includegraphics[scale=0.8]{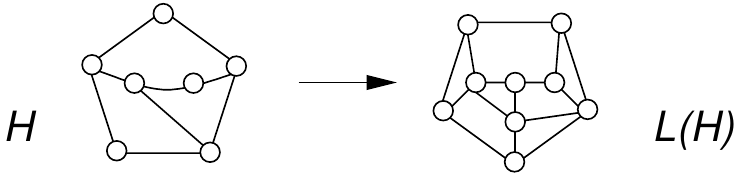}
\caption{A graph and its line graph.}
\label{grafos_line}
\end{center}
\end{figure}

\subsection{A semi-definite relaxation and $N_+$-perfect graphs} 

In the early nineties, Lov\'{a}sz and
Schrijver introduced the PSD-operator $N_+$ which, applied over the edge relaxation $\estab(G)$, generates the positive semi-definite relaxation $N_+(G)$ stronger than $\thbo(G)$ \cite{LovaszSchrijver1991}.
We denote by $\e_0, \e_1, \dots, \e_n$ the vectors of the canonical basis of $\R^{n+1}$ (where  the first coordinate is indexed zero), $\uno$ the vector with all components equal to 1 and $S_+^{n}$ the space of $(n \times n)$ symmetric and positive  semidefinite matrices with real entries.   
Given a convex set $K$ in $[0,1]^n$, let
\[
 \cone(K)= \left\{\left(\begin{array}{c} x_0\\ \x \end{array}
 \right) \in \R^{n+1}: \x=x_0 \y; \;\;  \y \in K \right\}.
 \]
Then, we define the 
set $ M_{+}(K) = $
 \begin{eqnarray*}
 \left\{Y \in S_+^{n+1}: \right.
 & & Y\e_0 = \diag (Y),\\
 & & Y\e_i \in \cone(K), \\ 
 & & \left. Y (\e_0 - \e_i) \in \cone(K), \; 
 i=1,\dots,n \right\},
 \end{eqnarray*}
where $\diag (Y)$ denotes the vector whose $i$-th entry is $Y_{ii}$, for every $i=0,\dots,n$.  
Projecting this lifting back to the space $\R^n$ results in $N_+(K) = $
 \[
 \set{ \x \in [0,1]^n : \left(\begin{array}{c} 1\\ \x \end{array}
 \right)= Y \e_0, \mbox{ for some } Y \in M_{+}(K)}.
 \]
In \cite{LovaszSchrijver1991}, Lov\'{a}sz and Schrijver proved that $N_+(K)$ is a relaxation of  the convex hull of integer solutions in $K$ and that $N_+^n(K)=\conv(K\cap \{0,1\}^n)$, where $N_+^0(K)=K$ and $N_+^k(K)=N_+(N_+^{k-1}(K))$ for 
$k\geq 1$. 
In this work we focus on the behaviour of a single application of the $N_+$-operator to the edge relaxation of the stable set polytope of a graph.

In order to simplify the notation we write $N_+(G)=N_+(\fra(G))$. 
In \cite{LovaszSchrijver1991} it is shown that 
\[\stab(G) \subseteq N_+(G) \subseteq \thbo(G) \subseteq \qstab(G).\]
As it holds for perfect graphs, the stable set problem can be solved in polynomial time for the class of graphs for which $N_+(G)=\stab(G)$. We will call these graphs \emph{$N_+$-perfect}. 
A graph $G$ that is not $N_+$-perfect is called \emph{$N_+$-imperfect}.

In \cite{BENT2011}, the authors look for a characterization of $N_+$-perfect graphs similar to the above characterizations for perfect graphs. 
More precisely, they intend to find an appropriate polyhedral relaxation of $\stab(G)$ playing the role of $\qstab(G)$ in \eqref{equivalencias}. 
Following this line, the following conjecture has been recently proposed in \cite{BENT2013}:

\begin{conjecture}[\cite{BENT2013}] \label{conjecture}[$N_+$-Perfect Graph Conjecture]
The stable set polytope of every $N_+$-perfect graph can be described by facet-defining inequalities with near-bipartite support. 
\end{conjecture}
\emph{Near-bipartite} graphs,
defined in \cite{Shepherd1995}, are those graphs where removing all neighbors of an arbitrary node and the node itself leaves the resulting graph bipartite.
Antiwebs and complements of line graphs are examples of near-bipartite graphs. 
Again from results in \cite{LovaszSchrijver1991}, we know that graphs for which every facet-defining inequality of $\stab(G)$ has a near-bipartite support is $N_+$-perfect. 
Thus, Conjecture \ref{conjecture} states that these graphs are the only $N_+$-perfect graphs. 
In particular, near-bipartite and a-perfect graphs are $N_+$-perfect. 

In addition, it can be proved that every subgraph of an $N_+$-perfect graph is also $N_+$-perfect.
This motivates the definition of \emph{minimally} $N_+$-\emph{imperfect graphs} as these $N_+$-imperfect graphs whose proper induced subgraphs are all $N_+$-perfect. 

\begin{figure}
\begin{center}
\includegraphics[scale=1.0]{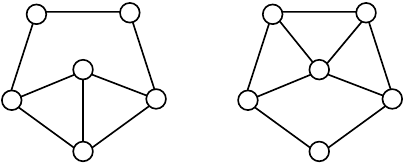}
\caption{The graphs $G_{LT}$ (on the left) and $G_{EMN}$ (on the right).}
\label{grafos6}
\end{center}
\end{figure}
In \cite{EMN2006} and \cite{LiptakTuncel2003} it was proved that all
the imperfect graphs with at most $6$ nodes are $N_+$-perfect,
except the two imperfect near-perfect graphs depicted in Figure \ref{grafos6}.
The graph on the left is denoted by $G_{LT}$ and the other one is denoted by $G_{EMN}$. 
So, $G_{LT}$ and $G_{EMN}$ are the two smallest minimally $N_+$-imperfect graphs. 
Characterizing all minimally $N_+$-imperfect graphs within a certain graph class can be a way to attack Conjecture \ref{conjecture} for this class. 
That way,  Conjecture \ref{conjecture} has been already verified for near-perfect graphs by \cite{BENT2011}, 
fs-perfect graphs (where the only facet-defining subgraphs are cliques and the graph itself) by \cite{BENT2013}, and 
webs (the complements of antiwebs) by \cite{EN2014}. 

In this contribution, we verify Conjecture \ref{conjecture} for line graphs. 
For that, we present three infinite families of 
$N_+$-imperfect line graphs (Section 2) and show that all facet-defining subgraphs of a line graph different from cliques and odd holes contain one of these 
$N_+$-imperfect line graphs (Section 3). 
Finally, we notice that the graphs in the three presented families are minimally 
$N_+$-imperfect and, in fact, the only minimally $N_+$-imperfect line graphs. 
We close with some concluding remarks and lines of further research.

\section{Three families of $N_+$-imperfect line graphs}

In this section, we provide three infinite families of 
$N_+$-imperfect line graphs. 
For that, we apply an operation preserving $N_+$-imperfection to the two smallest $N_+$-imperfect graphs $G_{LT}$ and $G_{EMN}$ (note that both graphs are line graphs).

In \cite{LiptakTuncel2003}, the \emph{stretching} of a node $v$ is introduced as follows: 
Partition its neighborhood $N(v)$ into two nonempty, disjoint sets $A_1$ and $A_2$ (so $A_1 \cup A_2 = N(v)$, and $A_1 \cap A_2 = \emptyset$). 
A stretching of $v$ is obtained by replacing $v$ by two adjacent nodes $v_1$ and $v_2$, joining $v_i$ with every node in $A_i$ for $i \in \{1, 2\}$, and either subdividing the edge $v_1 v_2$ by one node $w$ or subdividing every edge between $v_2$ and $A_2$ with one node. 
Moreover, it is shown in \cite{LiptakTuncel2003} that the stretching of a node preserves $N_+$-imperfection. 

For our purpose, we will use the stretching of node $v$ in the case of subdividing the edge $v_1 v_2$ by one node $w$.
If $|A_1|=1$ or $|A_2|=1$ the stretching corresponds to the \emph{3-subdivision} of an edge (that is, when the edge $v_1 v_2$ 
is replaced by a path of length 3).

We next establish a connection between subdivisions of edges in a graph $H$ and stretchings of nodes in its line graph $L(H)$. 
Let $G$ be the line graph of $H$ and consider an edge $e = u_1 u_2$ in $H$ together with its corresponding node $v$ in $G$. 
If $e$ is a simple edge of $H$ (i.e., if there is no edge parallel to $e$ in $H$), then the neighborhood $N(v)$ of its corresponding node $v$ in $G$ partitions into two cliques $U_1$ and $U_2$ representing the edges in $H$ incident to $e$ in $u_1$ and $u_2$, respectively. 
Accordingly, we call a stretching of a node $v$ in a line graph \emph{canonical} if these cliques $U_1$ and $U_2$ are selected as partition of $N(v)$. 

For illustration, see Figure \ref{fig_stretching} which shows the graph $C_5 + c$ (a $5$-hole with one chord $c$), the graph $C_5 + E_3$ (obtained from $C_5 + c$ by subdividing $c$ into a path $E_3$ of length 3), and their line graphs, where $L(C_5 + E_3)$ results from $L(C_5 + c)$ by a canonical stretching of the node corresponding to $c$. 

\begin{figure}
\begin{center}
\includegraphics[scale=1.0]{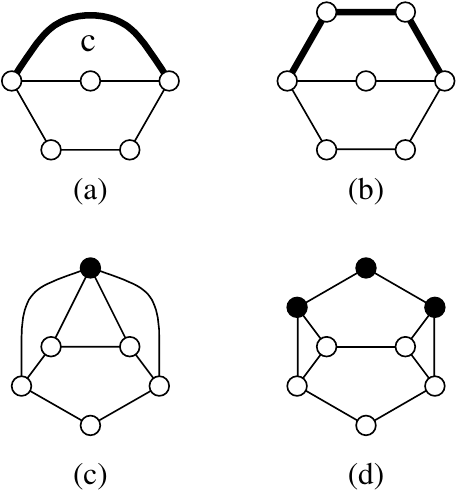}
\caption{This Figure shows (a) the graph $C_5 + c$ (a $5$-hole with one chord $c$), (b) the graph $C_5 + E_3$ (obtained from $C_5 + c$ by a 3-subdivision of $c$), (c) the line graph $L(C_5 + c)$, and (d) the line graph $L(C_5 + E_3)$, where $L(C_5 + E_3)$ results from $L(C_5 + c)$ by a canonical stretching of the (black-filled) node corresponding to $c$.}
\label{fig_stretching}
\end{center}
\end{figure}

In fact, we have in general:

\begin{lemma}\label{lem_stretching}
Let $e$ be a simple edge in $H$ and $v$ be the corresponding node in its line graph $G$. 
A 3-subdivision of $e$ in $H$ results in a canonical stretching of $v$ in $G$.
\end{lemma}

\begin{proof}
Let $e = u_1 u_2$ be a simple edge of $H$ and $H'$ be the graph obtained from $H$ by replacing $e$ by the path $u_1, u_1', u_2', u_2$. 

The line graph $L(H')$ contains a node $v_1$ representing the edge $u_1 u_1'$ of $H'$, a node $w$ corresponding to the edge $u_1' u_2'$ of $H'$, and a node $v_2$ for the edge $u_2' u_2$ of $H'$. 

In $L(H')$, $v_1$ is adjacent to $w$ and to a clique $U_1$ corresponding to all edges different from $u_1 u_1'$ incident to $u_1$, $w$ has only $v_1$ and $v_2$ as neighbors, and $v_2$ is adjacent to $w$ and to a clique $U_2$ corresponding to all edges different from $u_2' u_2$ incident to $u_2$. 

All other nodes and adjacencies in $L(H')$ are as in $L(H)$, 
hence $L(H')$ corresponds exactly to the graph obtained from $L(H)$ by the canonical stretching of the node $v$ representing the edge $e$.
\end{proof}
This enables us to show:

\begin{theorem}
\label{thm_N+imperfect_line}
A line graph $L(H)$ is $N_+$-imperfect if $H$ is 
  \begin{itemize}
  \item an odd hole with one double edge;
  \item an odd hole with one chord;
  \item an odd hole with one odd path attached to non-adjacent nodes of the hole.
  \end{itemize}  
\end{theorem}

\begin{proof}
Let $C_{2k+1}+d$ (resp. $C_{2k+1}+c$, resp. $C_{2k+1}+E_{\ell}$) denote the graph obtained from an odd hole $C_{2k+1}$ with $k \geq 2$ by adding one edge $d$ parallel to an edge of the hole (resp. adding one chord $c$ to the hole, resp. attaching one path of length $\ell$ to two non-adjacent nodes of the hole).

To establish the $N_+$-imperfection of the three families, we first observe that the two minimally $N_+$-imperfect graphs $G_{LT}$ and $G_{EMN}$ are line graphs: indeed, we have $G_{LT} = L(C_5+d)$ and $G_{EMN} = L(C_5+c)$.

Clearly, the graph $C_{2k+3}+d$ can be obtained from $C_{2k+1}+d$ by 3-subdivision of a simple edge (not being parallel to $d$) of the hole. 
Thus, any odd hole with one double edge can be obtained from $C_{5}+d$ by repeated 3-subdivisions of simple edges. 

According to Lemma \ref{lem_stretching}, their line graphs are obtained by repeated canonical stretchings of $G_{LT}$, which yields the first studied family of graphs. 

Analogously, $C_{2k+3}+c$ can be obtained from $C_{2k+1}+c$ by 3-subdivision of an edge of the hole. 
Thus, any odd hole with one (short or long) chord can be obtained from $C_{5}+c$ by repeated 3-subdivisions of edges different from $c$. 
Moreover, applying repeated 3-subdivisions of the chord $c$ yields graphs $C_{2k+1}+E_{\ell}$ where $E_{\ell}$ is a path of arbitrary odd length $\ell$  attached to two non-adjacent nodes of the hole at arbitrary distance. 

According to Lemma \ref{lem_stretching}, their line graphs are obtained by repeated canonical stretchings of $G_{EMN}$, which yields the two remaining families of graphs. 

That $G_{LT}$ and $G_{EMN}$ are minimally $N_+$-imperfect and canonical stretchings preserve $N_+$-imperfection completes the proof.
\end{proof}

\section{Characterizing $N_+$-perfect line graphs}

A combination of results by Edmods~\cite{Edmonds1965} and Edmonds \& Pulleyblank~\cite{EdmondsPulleyblank1974} about the matching polytope implies the following description of the stable set polytope of line graphs.

\begin{theorem}[Edmods~\cite{Edmonds1965}, Edmonds \& Pulleyblank~\cite{EdmondsPulleyblank1974}]
\label{thm_STAB_LineGraph}
If $G$ is the line graph of a graph $H$, then $\stab(G)$ is decribed by nonnegativity constraints, maximal clique constraints, and rank constraints  
\begin{equation}
\label{Eq_hypomatch-constraints}
x(L(H')) \le \frac{|V(H')|-1}{2}
\end{equation}
associated with the line graphs of 2-connected hypomatchable induced subgraphs $H' \subseteq H$.
\end{theorem}

A graph $H$ is \emph{hypomatchable} if, for all nodes $v$ of $H$, the subgraph $H-v$ admits a perfect matching (i.e., a matching meeting all nodes) 
and is \emph{2-connected} if it remains connected after removing an arbitrary node.

Due to a result by Lov\'asz~\cite{Lovasz1972}, a graph $H$ is hypomatchable if and only if there is a sequence $H_0, H_1, \dots, H_k = H$ of graphs such that $H_0$ is a chordless odd cycle and for $1 \leq i \leq k$, 
$H_i$ is obtained from $H_{i-1}$ by adding an odd path $E_i$ that joins two (not necessarily distinct) nodes of $H_{i-1}$ and has all internal nodes outside $H_{i-1}$.
The odd paths $E_i = H_i - H_{i-1}$ are called \emph{ears} for $1 \leq i \leq k$ and 
the sequence $H_0, H_1, \dots, H_k = H$ an \emph{ear decomposition} of $H$, see Fig. \ref{ear_decomp} for illustration. 
Moreover, we call an ear of length at least three \emph{long}, and \emph{short} otherwise.

\begin{figure}[h]
\begin{center}
\includegraphics[scale=0.8]{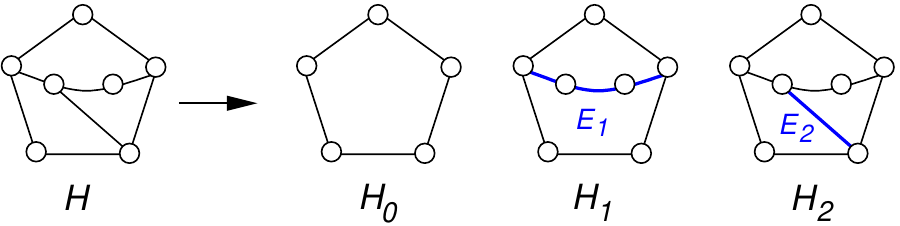}
\caption{A graph and an ear decomposition.}
\label{ear_decomp}
\end{center}
\end{figure}

Hypomatchable graphs have an odd number of nodes, 
are non-bipartite and connected, but 
neither necessarily 2-connected (since an ear $E_i$ may be attached to a single node of $H_{i-1}$) 
nor simple (since a short ear $E_i$ may become an edge parallel to one edge of $H_{i-1}$). 

However, if $H$ is 2-connected,  
Cornu\'ejols \& Pulleyblank~\cite{CornuejolsPulleyblank1983} proved 
that $H$ admits an ear decomposition $H_0, H_1, \dots, H_k = H$ 
with $H_i$ 2-connected for every $0 \leq i \leq k$.  
If, in addition, $H$ has at least five nodes, Wagler~\cite{Wagler2000} later proved 
that $H$ admits an ear decomposition $H_0, H_1, \dots, H_k = H$ 
where $H_0$ has at least five nodes and $H_i$ is 2-connected for every $0 \leq i \leq k$.   

Since the latter result is a key property for our argumentation, we give its proof for the sake of completeness:

\begin{lemma}[Wagler~\cite{Wagler2000}] Let $H$ be a 2-connected hypomatchable graph and $|V (H)| \geq 5$. 
Then there is an ear decomposition $H_0, H_1, \dots, H_k = H$ of $H$ such that each $H_i$ is 2-connected and $H_0$ is an odd cycle of length at least five.
\end{lemma}

\begin{proof} Since $H$ is 2-connected, it admits an ear decomposition \\ $H_0, H_1, \dots, H_k = H$ with $H_i$ 2-connected for $0 \leq i \leq k$ by \cite{CornuejolsPulleyblank1983}. 
We are ready if $H_0$ is an odd cycle of length $\geq$ 5, hence assume, for the sake of contradiction, that $H_0$ is a triangle. 

From $|V (H)| > 3$ follows that there is an ear with at least three edges, let $i \in \{1, . . . , k\}$ be the smallest index such that $E_i$ has length $\geq 3$. 
Then $V (H_{i-1}) = V (H_0)$ holds and $E_i$ has two distinct nodes $v, v' \in V (H_0)$ as endnodes
(since $H_i$ is 2-connected). Hence $(H_0 - vv') \cup E_i, vv', E_1, . . . , E_{i-1}$ is an ear decomposition of $H_i$ starting with an odd cycle of length $\geq 5$ and defining only 2-connected intermediate graphs. The ears $E_{i+1}, . . . , E_k$ complete this ear decomposition to the studied decomposition of $H$. 
\end{proof}

With the help of these results, 
we can give the following characterization of 2-connected hypomatchable graphs:

\begin{theorem}
\label{thm_hypomatchable}
Let $H$ be a 2-connected hypomatchable graph, then exactly one of the following conditions is true:
\begin{itemize}
\item $H$ has only three nodes;
\item $H$ equals an odd hole;
\item $H$ contains one of the following subgraphs:
  \begin{itemize}
  \item an odd hole with one double edge;
  \item an odd hole with one chord;
  \item an odd hole with one long ear, attached to non-adjacent nodes of the hole.
  \end{itemize}  
\end{itemize}  
\end{theorem}

\begin{proof}
Consider a 2-connected hypomatchable graph $H$ and distinguish the following cases. 


If $H$ has more than three nodes, then $H$ admits an ear decomposition $H_0, H_1, \dots, H_k = H$ 
where $H_0$ has at least five nodes and $H_i$ is 2-connected for every $0 \leq i \leq k$ by Wagler~\cite{Wagler2000}. 

If $H = H_0$, then $H$ equals an odd hole. 

If $H \neq H_0$, then $H_1$ equals one of the above mentioned graphs:
  \begin{itemize}
  \item an odd hole with one double edge (if $E_1$ is a short ear attached to adjacent nodes of $H_{0}$);
  \item an odd hole with one chord (if $E_1$ is a short ear attached to non-adjacent nodes of $H_{0}$ 
 or if $E_1$ is a long ear attached to adjacent nodes of $H_{0}$);
  \item an odd hole with one long ear $E_1$, attached to non-adjacent nodes of $H_{0}$ (otherwise).
  \end{itemize} 
(Recall that $E_1$ cannot be attached to a single node of $H_{0}$ since $H_1$ is 2-connected.)
\end{proof}

Combining Theorem~\ref{thm_hypomatchable} and Theorem~\ref{thm_N+imperfect_line}, we can further prove:

\begin{theorem}
\label{thm_line_hypomatchable}
If $H$ is a 2-connected hypomatchable graph, then $L(H)$ is either a clique, an odd hole, or $N_+$-imperfect.
\end{theorem}

\begin{proof}
Consider a 2-connected hypomatchable graph $H$ and distinguish the following cases. 

If $H$ has only three nodes, then $H$ has an ear decomposition \\ $H_0, H_1, \dots, H_k = H$ where $H_0$ is a triangle and all ears are short, becoming edges parallel to one edge of $H_{0}$. 
In this case, $L(H)$ is clearly a clique.

If $H$ has more than three nodes, then $H$ admits an ear decomposition $H_0, H_1, \dots, H_k = H$ 
where $H_0$ is an odd hole with at least five nodes and $H_i$ is 2-connected for every $0 \leq i \leq k$ by Wagler~\cite{Wagler2000}. 

If $H = H_0$, then $H$ equals an odd hole and $L(H)$ is clearly an odd hole, too. 

If $H \neq H_0$, then $H_1$ equals one of the graphs from Theorem~\ref{thm_hypomatchable}:
  \begin{itemize}
  \item an odd hole with one double edge;
  \item an odd hole with one chord;
  \item an odd hole with one long ear, attached to non-adjacent nodes of $H_{0}$.
  \end{itemize} 
According to Theorem~\ref{thm_N+imperfect_line}, the line graph $L(H_1)$ is $N_+$-imperfect, hence $L(H)$ is $N_+$-imperfect.
\end{proof}

Combining Theorem~\ref{thm_line_hypomatchable} and the description of stable set polytopes of line graphs from Theorem~\ref{thm_STAB_LineGraph} further yields:

\begin{corollary} 
\label{Cor_line_hypomatchable}
A line graph $L(H)$ is $N_+$-perfect if and only if all 2-connected hypomatchable induced subgraphs $H' \subseteq H$ have either only three nodes or are odd holes.
\end{corollary} 

Thus, the definition of h-perfect graphs finally implies the following characterization of $N_+$-perfect line graphs:

\begin{corollary} 
\label{Cor_line_N+perfect}
A line graph is $N_+$-perfect if and only if it is h-perfect.
\end{corollary} 

Since both the class of line graphs and the class of $N_+$-perfect graphs are hereditary (i.e., closed under taking induced subgraphs), we can derive from a characterization of $N_+$-perfect line graphs also a characterization of minimally $N_+$-imperfect line graphs.

In fact, combining the assertions from Theorem \ref{thm_N+imperfect_line}, 
Theorem \ref{thm_STAB_LineGraph} by Edmonds~\cite{Edmonds1965} and Edmonds \& Pulleyblank~\cite{EdmondsPulleyblank1974}, 
Theorem \ref{thm_hypomatchable} and Corollary \ref{Cor_line_hypomatchable}, 
we can even conclude that the graphs in the three families presented in Theorem \ref{thm_N+imperfect_line} are minimally $N_+$-imperfect and, in fact, the only minimally $N_+$-imperfect line graphs. This leads to the following characterization:

\begin{corollary}
A line graph $L(H)$ is minimally $N_+$-imperfect if and only if $H$ is an odd hole with one ear 
attached to distinct nodes of the hole.
\end{corollary}

\begin{proof}
On the one hand, each of the line graphs $L(H)$ with 
$H \in \{ C_{2k+1}+d, C_{2k+1}+c, C_{2k+1}+E_{\ell} \}$ from Theorem \ref{thm_N+imperfect_line} has, 
by construction, the property that $H$ is a 2-connected hypomatchable graph and admits an ear decomposition $H_0, H_1 = H$ with $H_0 = C_{2k+1}, k\geq 2$. 

Removing any edge $e$ from $H$ yields a graph $H-e$ which is 
either an odd hole (if $e \in \{ c,d \}$) 
or else not hypomatchable anymore. 
In all cases, $H-e$ is bipartite or contains at most one odd cycle such that $L(H-e)$ has cliques and odd holes as only facet-defining subgraphs. 

Hence, $L(H)$ is $N_+$-imperfect by Theorem \ref{thm_N+imperfect_line}, but all proper induced subgraphs are h-perfect (and, thus, $N_+$-perfect). 
This implies that $L(H)$ is minimally $N_+$-imperfect for all graphs within the 3 studied families. 

On the other hand, any minimally $N_+$-imperfect graph needs to have a full-support facet. 
Thus, any minimally $N_+$-imperfect line graph $L(H)$ is the line graph of a 2-connected hypomatchable graph $H$ (by Theorem \ref{thm_STAB_LineGraph}) which has more than 3 nodes and is different from an odd hole (by Theorem \ref{thm_hypomatchable}). 

Then $H$ admits an ear decomposition $H_0, H_1, \dots, H_k = H$ with $k \geq 1$ where $H_0$ has at least five nodes and $H_i$ is 2-connected for every $0 \leq i \leq k$ by Wagler~\cite{Wagler2000}.  

In fact,  $k = 1$ follows since $H_1 \in \{ C_{2k+1}+d, C_{2k+1}+c, C_{2k+1}+E_{\ell} \}$ holds as $H_1$ is 2-connected and, thus, already $L(H_1)$ is $N_+$-imperfect by Theorem \ref{thm_N+imperfect_line} (note: $L(H)$ contains $L(H_1)$, is $N_+$-imperfect and minimally $N_+$-imperfect only if $L(H) = L(H_1)$). 

Finally, $H_1$ has the stated property: it is an odd hole $H_0$ with one ear 
attached to distinct nodes of $H_0$.
\end{proof}

\section{Conclusion and further results}

In this contribution, 
we addressed the problem to verify Conjecture \ref{conjecture} for line graphs. 
For that, we presented three infinite families of 
$N_+$-imperfect line graphs (Section 2) and showed that all facet-defining subgraphs of a line graph different from cliques and odd holes contain one of these 
$N_+$-imperfect line graphs (Section 3). 
Since cliques and odd holes are clearly near-bipartite, Corollary \ref{Cor_line_N+perfect} shows that Conjecture \ref{conjecture} is true for line graphs. 

In the following, we will discuss a reformulation of Conjecture \ref{conjecture}.

As superclass of a-perfect graphs, \emph{joined a-perfect graphs} where introduced in \cite{CPW_2009} as those graphs whose only facet-defining subgraphs are complete joins of a clique and prime antiwebs (an antiweb $A^k_n$ is \emph{prime} if $k+1$ and $n$ are relatively prime integers). 
The inequalities obtained from complete joins of antiwebs, called \textit{joined antiweb constraints}, are of the form 
\[
\sum_{i \leq k} \frac{1}{\alpha(A_i)} x(A_i) + x(Q) \leq 1, 
\]
where $A_1, \ldots, A_k$ are different antiwebs and $Q$ is a clique (note that the inequalities are scaled to have right hand side equal to 1). 
We denote by $\astab^*(G)$ the linear relaxation of $\stab(G)$ obtained by all joined antiweb constraints. 
Then, a graph $G$ is joined a-perfect if $\stab(G)$ equals $\astab^*(G)$.

In particular, Shepherd \cite{Shepherd1995} showed that the stable set polytope of a near-bipartite graph has only facet-defining inequalities associated with complete joins of a clique and prime antiwebs. Thus, every near-bipartite graph is joined a-perfect (but the converse is not true since there exist perfect graphs that are not near-bipartite, see Figure \ref{fig_NotNearBip} for an example). 

\begin{figure}
\begin{center}
\includegraphics[scale=1.0]{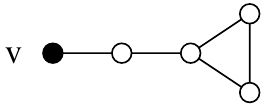}
\caption{A perfect (and, thus, joined a-perfect) graph with a node $v$ such that removing $v$ and its neighbor(s) leaves a non-bipartite graph.}
\label{fig_NotNearBip}
\end{center}
\end{figure}

Moreover, Conjecture \ref{conjecture} identifies $N_+$-perfect graphs and graphs for which its stable set polytope can be described by inequalities with near-bipartite support. It is known that, given a graph $G$, every facet defining inequality of $\stab(G)$ with support graph $G'$ is a facet defining inequality of $\stab(G')$. Then, again by Shepherd's results \cite{Shepherd1995}, those graphs for which its stable set polytope can be described by inequalities with near-bipartite support are joined a-perfect graphs. 

Taking this into account, Conjecture \ref{conjecture} can be reformulated as follows:

\begin{conjecture}
Every $N_+$-perfect graph is joined a-perfect. 
\end{conjecture}

The results of Lov\'asz and Schrijver \cite{LovaszSchrijver1991} prove that joined a-perfect graphs are $N_+$-perfect, thus, the conjecture states that both graph classes coincide and that $\astab^*(G)$ shall be the studied polyhedral relaxation of $\stab(G)$ playing the role of $\qstab(G)$ in \eqref{equivalencias}. 

In particular, in this paper we have proved that every $N_+$-perfect line graph is h-perfect. Then, combining these results, we obtain that a line graph is joined a-perfect if and only if it is h-perfect. However, it seems natural to look for a proof of the latter result independent of the $N_+$-operator.  

\begin{theorem}
A line graph is joined a-perfect if and only if it is h-perfect.
\end{theorem}

\begin{proof}
A joined a-perfect graph has as only facet-defining subgraphs complete joins of a clique and prime antiwebs. 

Following the same argumentation as in \cite{Shepherd1995} that odd antiholes are the only prime antiwebs in complements of line graphs, we see that odd holes are the only prime antiwebs in line graphs. 

Analogous arguments as in \cite{Wagler2004_4OR,Wagler2005_4OR} yield that in a line graph, no complete join of a clique and odd holes or of two or more odd holes can occur: every such complete join would particularly contain an odd wheel $W_{2k+1}$, i.e., the complete join of a single node and an odd hole $C_{2k+1}$. 

The $W_5$ is one of the minimal forbidden subgraphs of line graphs by Beineke \cite{Beineke1968}, larger odd wheels contain a claw, another minimal forbidden subgraph of line graphs by \cite{Beineke1968}. 

Thus, the only remaining facet-defining subgraphs in a joined a-perfect line graph are cliques and odd holes. 

Conversely, an h-perfect line graph is clearly joined a-perfect.
\end{proof}


Our future lines of further research include
\begin{itemize}
\item to look for new families of graphs where the conjecture holds (e.g., by characterizing the minimally $N_+$-imperfect graphs within the class);
\item to find new subclasses of $N_+$-perfect or joined a-perfect graphs.
\end{itemize}
In all cases, the structural results would have algorithmic consequences since the stable set problem could be solved in polynomial time for the whole class or its intersection with $N_+$-perfect or joined a-perfect graphs by optimizing over $N_+(G)$.

\section*{Acknowledgment}

This work was supported by an ECOS-MINCyT cooperation France-Argentina, A12E01, PIP-CONICET 0241, PICT-ANPCyT 0361, PID-UNR 415 and 416.





\begin{thebibliography}{10}



\bibitem{Beineke1968}
 Beineke L.W. (1968). 
Derived graphs and digraphs, 
In: \textit{Beitr{\"a}ge zur Graphentheorie}, H.~Sachs, H.~Voss und H.~Walther (Hrg.), 
Teubner Verlag, Leipzig, 17-33.


\bibitem{BENT2011} Bianchi S., Escalante M., Nasini G.,  Tun\c{c}el L. (2011). 
Near-perfect graphs with polyhedral $N_+(G)$,
\textit{Electronic Notes in Discrete Mathematics} 37, 393-398.

\bibitem{BENT2013} 
Bianchi S., Escalante M., Nasini G., Tun\c{c}el  L. (2013), 
Lov\'asz-Schrijver PSD-operator and a superclass of near-perfect graphs, 
\textit{Electronic Notes in Discrete Mathematics}44, 339-344.


\bibitem{ChudnovskyEtAl2006}
Chudnovsky M., Robertson N., Seymour  P., Thomas R. (2006). 
The Strong Perfect Graph Theorem, 
\textit{Annals of Mathematics} 164, 51-229.


\bibitem{Chvatal1975}
Chv\'atal V. (1975).
On Certain Polytopes Associated with Graphs,
\textit{J.~Combin.~Theory (B)} 18, 138-154. 

\bibitem{CornuejolsPulleyblank1983}
Cornu\'ejols G., Pulleyblank W.R. (1983). Critical Graphs, Matchings, and Tours of a Hier\-archy of Relaxations for the  Traveling Salesman Problem,
\textit{Combinatorica} 3, 35-52.

\bibitem{CPW_2009}
 Coulonges S., P{\^e}cher  A., Wagler  A. (2009). 
Characterizing and bounding the imperfection ratio for some classes of graphs, 
\textit{Mathematical Programming A} 118, 37-46.

\bibitem{Edmonds1965}
J.R.~Edmonds, 
``Maximum Matching and a Polyhedron with (0,1) Vertices,'' 
\textit{J. Res. Nat. Bur. Standards}, vol.~69B, pp.~125--130, 1965.

\bibitem{EdmondsPulleyblank1974}
Edmonds J.R., Pulleyblank W.R. (1974) 
Facets of 1-Matching Polyhedra,
In: C. Berge and D.R. Chuadhuri (eds.) \textit{Hypergraph Seminar}. 
Springer, 214-242. 


\bibitem{EMN2006}
Escalante M., M.S. Montelar M.S., Nasini G. (2006). 
Minimal $N_+$-rank graphs: Progress on Lipt\'ak and Tun\c{c}el's conjecture, 
\textit{Operations Research Letters} 34, 639-646.

\bibitem{EN2014}
Escalante M., Nasini G. (2014).
Lov\'asz and Schrijver $N_+$-relaxation on web graphs,
\textit{Lecture Notes in Computer Sciences} 8596 (Special Issue ISCO 2014), 221-229.
 
\bibitem{GrotschelLovaszSchrijver1981} 
Gr\"otschel M., Lov\'asz L.,  Schrijver A. (1981). 
The Ellipsoid Method and its Consequences in Combinatorial Optimization, 
\textit{Combinatorica} 1, 169-197.

\bibitem{GrotschelLovaszSchrijver1988}
Gr\"otschel M., Lov\'asz  L., Schrijver  A. (1988).
{\it Geometric Algorithms and Combinatorial Optimization.} Springer-Verlag.

\bibitem{LiptakTuncel2003}
Lipt\'ak L., Tun\c{c}el L. (2003).
Stable set problem and the lift-and-project ranks of graphs, 
\textit{Math. Programming A} 98, 319-353.

\bibitem{Lovasz1972} 
Lov\'asz L. (1972).
A Note on Factor-critical Graphs,
\textit{Studia Sci.~Math.~Hungar.} 7, 279-280.

\bibitem{Lovasz1979}
Lov\'{a}sz L. (1979), On the Shannon capacity of a graph, \textit{ Transactions} 25, 1-7.

\bibitem{LovaszSchrijver1991}
Lov\'asz L., Schrijver A. (1991). 
Cones of matrices and set-functions and 0-1 optimization, 
\textit{SIAM J. on Optimization}  1, 166-190.

\bibitem{Shepherd1994}
Shepherd F.B. (1994).
Near-Perfect Matrices,
\textit{Math. Programming} 64, 295-323.

\bibitem{Shepherd1995}
Shepherd F.B. (1995)
Applying Lehman's Theorem to Packing Problems, 
\textit{Math. Programming} 71, 353-367.


\bibitem{Wagler2000}
Wagler A. (2000). 
Critical Edges in Perfect Graphs. 
\textit{PhD thesis, TU Berlin and Cuvillier Verlag G\"ottingen}.

\bibitem{Wagler2004_4OR}
Wagler A.(2004). 
Antiwebs are rank-perfect,
\textit{4OR} 2, 149-152.

\bibitem{Wagler2005_4OR}
Wagler A.(2005). 
On rank-perfect subclasses of near-bipartite graphs,
\textit{4OR} 3, 329-336.

\end{thebibliography}
%

\end{document}